\documentclass[reqno]{amsart}
\usepackage[english]{babel}
\usepackage{graphicx,subfigure}
\usepackage{amsmath,amssymb}
\usepackage{appendix}
\usepackage{color}
\definecolor{oneblue}{rgb}{0,0.0,0.75}

\usepackage{fancyhdr}

\makeatletter
\newcommand{\sech}{\mathop{\operator@font sech}}
\newcommand{\sign}{\mathop{\operator@font sign}}
\makeatother

\newtheorem{theorem}{Theorem}[section]
\newtheorem{proposition}{Proposition}[section]

\newtheorem{remark}{Remark}[section]
\numberwithin{equation}{section}

\begin{document}

\title[]{A numerical method for two-dimensional Boussinesq-Full Dispersion systems modelling internal wave propagation}


\author[A. Duran]{Angel Duran}
\address{ Applied Mathematics Department,  University of
Valladolid, 47011 Valladolid, Spain}
\email{angeldm@uva.es}




\begin{abstract}
The paper is concerned with the numerical approximation of some two-dimensional asymptotic models for the propagation of internal waves in a two-layer system in the form of pde's for the interfacial deviation and the velocity. After the study of well-posedness of the periodic initial-value problem, the convergence of a semidiscrete approximation based on a spectral Fourier-Galerkin method is analyzed. The resulting semidiscrete system is then integrated numerically in time with the implicit midpoint rule.  Some properties of the fully discrete scheme are described and its performance is illustrated with some numerical experiments.
\end{abstract}

\maketitle

\section{Introduction}
The purpose of the paper is the introduction and analysis of a numerical method to approximate some internal wave models in two dimensions, derived in \cite{BLS2008} (see also \cite{S}), of the form
\begin{eqnarray}
(Id-b\Delta)\zeta_{t}+\frac{1}{\gamma}\nabla\cdot\left((1-\zeta){\bf v}_{\beta}\right)-\frac{1}{\gamma^{2}}|D|{\rm coth}(|D|)\nabla\cdot{\bf v}_{\beta}\nonumber&&\\
+\frac{1}{\gamma}\left(a-\frac{1}{\gamma^{2}}{\rm coth}^{2}(|D|)\right)\Delta\nabla\cdot{\bf v}_{\beta}&=&0,\label{BFD1a}\\
(Id-d\Delta)({\bf v}_{\beta})_{t}+(1-\gamma)\nabla\zeta-\frac{1}{2\gamma}\nabla |{\bf v}_{\beta}|^{2}+(1-\gamma)c\Delta\nabla\zeta&=&0.\label{BFD1b}
\end{eqnarray}
The system (\ref{BFD1a}), (\ref{BFD1b}) requires a detailed description in terms of modelling and mathematically. It is an approximation to the wave propagation in two dimensions along the interface of a two-layer system of inviscid, incompressible, homogeneous fluids with depths $d_{j}$ and densities $\rho_{j}, j=1,2$, with $\rho_{1}<\rho_{2}$, \cite{Lannes2,Duchene2021}. The top of the upper layer is not free, but an impenetrable bounding surface (rigid-lid condition) while the bottom of the lower layer is rigid as well, and flat. Defining
\begin{eqnarray*}
\gamma=\frac{\rho_{1}}{\rho_{2}}<1,
\end{eqnarray*}
as the density ratio, the function $\zeta=\zeta(x,y,t)$ in (\ref{BFD1a}), (\ref{BFD1b}) represents the interfacial deviation at $(x,y)$, from some rest position, at time $t$. The vector field ${\bf v}_{\beta}$ is of the form ${\bf v}_{\beta}=(Id-\beta\Delta)^{-1}{\bf v}$, where $\Delta$ is the Laplace operator, $\beta\geq 0$ is a modelling parameter, and ${\bf v}={\bf v}(x,y,t)$ represents the two components of the horizontal velocity of the wave. The usual symbol $\nabla$ denotes the gradient operator (in the horizontal $(x,y)$ variables), while the symbol $|D|$ is defined as $|D|=(-\Delta)^{1/2}$ and, using Fourier multiplier notation, is associated to the Fourier multiplier
\begin{eqnarray*}
\widehat{\left(|D|f\right)}({\bf k})=|{\bf k}|\widehat{f}({\bf k}),\quad {\bf k}\in \mathbb{R}^{2},
\end{eqnarray*}
being $\widehat{f}({\bf k})$ the Fourier transform of $f\in L^{2}(\mathbb{R}^{2})$ at ${\bf k}=(k_{x},k_{y})$ and $|{\bf k}|=\sqrt{k_{x}^{2}+k_{y}^{2}}$. In addition, the system (\ref{BFD1a}), (\ref{BFD1b}) is four-parameter, with constants $a,b,c,d$ satisfying, \cite{BLS2008}
\begin{eqnarray*}
a+b+c+d=\frac{1}{3}.
\end{eqnarray*}
The corresponding Euler equations for the two-layer model are shown to be consistent (in the sense precised in \cite{BLS2008}) with (\ref{BFD1a}), (\ref{BFD1b}) under the so-called Boussinesq-Full Dispersion (B-FD) physical regime. This can be described from the dimensionless parameters
\begin{eqnarray*}
\epsilon_{j}=\frac{a}{d_{j}},\quad \mu_{j}=\frac{d_{j}^{2}}{\lambda^{2}},\quad j=1,2,
\end{eqnarray*}
where $a$ and $\lambda$ are, respectively, the typical elevation and horizontal wavelength of the wave. The parameters $\epsilon_{j}, j=1,2$, account for the nonlinear effects on the model with respect to the $j$th layer, while $\mu_{j}, j=1,2$, measure the dispersive, linear effects. The B-FD regime is then defined by the conditions
\begin{eqnarray*}
\mu_{1}\sim\epsilon_{1}<<1,\quad \epsilon_{2}<<1,\quad \mu_{2}\sim 1,
\end{eqnarray*}
and the consistency of Euler's system has a precision $O(\epsilon^{3/2})$.

Linear well-posedness of the initial-value problem (ivp) for (\ref{BFD1a}), (\ref{BFD1b}) is aso established in \cite{BLS2008} from the dispersion relation
\begin{eqnarray}
\omega^{2}&=&|{\bf k}|^{2}(1-\gamma)(1-c|{\bf k}|^{2})\frac{g({\bf k})}{(1+b|{\bf k}|^{2})(1+d|{\bf k}|^{2})},\nonumber\\
g({\bf k})&=&\frac{1}{\gamma}\left(1-\frac{1}{\gamma}|{\bf k}|{\rm coth}|{\bf k}|-|{\bf k}|^{2}\left(a-\frac{1}{\gamma^{2}}{\rm coth}^{2}|{\bf k}|\right)\right),\label{BFD1d}
\end{eqnarray}
under the conditions $b,d\geq 0, a,c\leq 0$. On the other hand, nonlinear well-posedness is analyzed in \cite{A}. The results are summarized in Table \ref{BFD_t1}, in terms of the sign of the parameters $a,b,c,d$ and under linear well-posedness assumption (admissible systems), along with the corresponding Sobolev space where existence and uniqueness of solution of the ivp, locally in time, holds. 

\begin{table}[htbp]
\begin{tabular}{|c|c|c|c|c|c|}
    \hline
No.&$b$&$d$&$a$&$c$&Nonlinear well-posedness\\\hline\hline
1&$+$&$+$&$-$&$-$&$H^{s}\times H^{s}\times H^{s}, s>0$ \\\hline
2&$+$&$+$&$-$&$0$&$H^{s-1}\times H^{s}\times H^{s}, s>0$\\\hline
3&$+$&$+$&$0$&$-$&$H^{s}\times H^{s}\times H^{s}, s>0$ \\\hline
4&$+$&$+$&$0$&$0$&$H^{s-1}\times H^{s}\times H^{s}, s>0$\\\hline
5&$+$&$0$&$-$&$-$&$H^{s+1}\times H^{s}\times H^{s}, s>2$ \\\hline
6&$+$&$0$&$-$&$0$&$H^{s}\times H^{s}\times H^{s}, s>2$ \\\hline
7&$+$&$0$&$0$&$-$&$H^{s+1}\times H^{s}\times H^{s}, s>2$\\\hline
8&$+$&$0$&$0$&$0$&$H^{s}\times H^{s}\times H^{s}, s>2$\\\hline
9&$0$&$+$&$-$&$-$&$H^{s}\times H^{s+1}\times H^{s+1}, s>1$ \\\hline
10&$0$&$+$&$-$&$0$&$H^{s}\times H^{s+2}\times H^{s+2}, s>1$ \\\hline
11&$0$&$+$&$0$&$-$&$H^{s}\times H^{s+1}\times H^{s+1}, s>1$\\\hline
12&$0$&$+$&$0$&$0$&$H^{s}\times H^{s+2}\times H^{s+2}, s>1$\\\hline\hline
\end{tabular}
\caption{B-FD systems: Nonlinear well-posedness theory, \cite{A}.}
\label{BFD_t1}
\end{table}
As far as the numerical approximation of (\ref{BFD1a}), (\ref{BFD1b}) is concerned, the only reference that we are aware of is \cite{DDS3}, for the one-dimensional case, and the present paper extends some of its results to the 2D case, as explained below. As for the approximations of systems of Boussinesq type, for surface and internal wave models, Galerkin-Finite Elements discretizations in space with high-order, explicit Runge-Kutta (RK) methods in time are proposed to approximate some initial-boundary-value problems (ibvp) for Boussinesq systems for surface waves in \cite{AD2,ADM1,DMS2007,DMS3}. Error estimates for the semidiscrete (in space) schemes are derived. In 1D, the use of spectral methods, of Galerkin or collocation type, to discretize in space, with high-order explicit RK or Diagonally Implicit RK of composition type as time integrators, is proposed and analyzed in \cite{XRAA,DDS0}. For the 2D case, in the recent work \cite{D2026}, a spectral Fourier-Galerkin method for the periodic ivp of systems of Boussinesq-Boussinesq type for internal waves is introduced and the convergence is proved for several well-posed systems. The Implicit Midpoint Rule (IMR) is proposed as time integrator and the performance of the full discretization is illustrated with some numerical experiments. The aim of the present paper is the extension of the analysis of this scheme for the B-FD system  (\ref{BFD1a}), (\ref{BFD1b}). In addition, the geometric character, as time integrator, of the IMR, is also used in \cite{BuliX2018}, where for the 1D periodic ivp of Boussinesq BBM-BBM system for surface waves, a scheme based on a local discontinous Galerkin  methods, with two different choices of numerical fluxes, is introduced. The accuracy and well behaviour in long-term simulations of fully discrete schemes with IMR and a high-order Hamiltonian dissipative Strong Stability Preserving RK method are tested and compared numerically.

As mentioned before, two main references, \cite{DDS3,D2026}, are in the origin of the present papr and determine the structure and contents in the following sense:
\begin{itemize}
\item In order to investigate the use of spectral methods to simulate in 2D pde systems for internal wave propagation, the results obtained in \cite{D2026} for the 2D Boussinesq/Boussinesq systems provide a motivation to analyze the same approach to approximate the periodic ivp of (\ref{BFD1a}), (\ref{BFD1b}). The main difference here is the presence of nonlocal terms which, via the Fourier representation, justify the spectral approach and provide an insight into the analysis of convergence and general performance of the approximation.
\item In order to study the convergence of the spectral discretization in space, introduced in the present paper, a comparison with the results in 1D given in \cite{DDS3} must be made. We analyze the same groups of B-FD systems considered in 1D, all within the so-called weakly dispersive case, that is, when $b,d> 0$, namely:
\begin{itemize}
\item[(G1)] The BBM-BBM B-FD case: $a=c=0$.
\item[(G2)] The generic B-FD case: $a,c<0$.
\item[(G3)] B-FD systems with $a<0, c=0$ or $a=0,c<0$.
\end{itemize}
According to the results presented in \cite{DDS3} and in this paper, in the case of (G1), the error estimates are similar in 1D and 2D, but with some differences: in both cases, spectral convergence is attained for smooth exact solutions, although the 2D case requires more regularity from the beginning. We also emphasize that in the 2D case, the estimate is of optimal order for the approximation to both the interfacial deviation and the velocity in $L^{2}$ but suboptimal for the velocity components in $H^{1}$ error. The cases  (G2) and (G3) when $a=0, c<0$ show the most relevant discrepancies. For 1D, the error estimates are optimal in $L^{2}$ norm for the two variables, while in 2D, requiring more regularity for the exact solution, they are optimal in $H^{1}$ norm for the three variables. Finally, in the case (G3) when $a<0, c=0$, the same type of estimate as in (G2) is obtained in 1D, while in 2D the error is still of optimal order in $H^{1}$ for the approximation to the velocities, but suboptimal in $L^{2}$ norm for the approximation to the interfacial wave.

We notice that, following similar strategies for the proofs in 1D and 2D, the disparity in the results explained above is mainly due to the differences between one and two dimensions in some estimates for the projection errors and inverse inequalities  in the space of trigonometric polynomials.
\item
The choice and implementation of the fully discrete scheme is mainly justified by the good results shown in \cite{D2026} for a different model of internal wave propagation. Selecting the IMR as time integrator ensures stability and geometric properties, \cite{HLW2004}, that are key aspects to the simulation of some models of the family (\ref{BFD1a}), (\ref{BFD1b}), specially in the generic case (G2), but requiring, however, some computational disadvantages, typically due to the diagonally implicit character of the integrator. Although, as we will see in the paper, the classical fixed-point algorithm for the internal stage is convergent under reasonable conditions for the stepsizes, the iterative procedure yields a slower performance. In this sense, we notice that in cases of BBM-BBM type (G1), explicit time discretizations may be used as alternative, \cite{CDM2026}.
\end{itemize}
The structure of the paper is as follows. In section \ref{sec2}, the periodic ivp for  (\ref{BFD1a}), (\ref{BFD1b}) is introduced and from the study in \cite{A} for the ivp, well-posedness of the periodic problem is analyzed. Section \ref{sec3} is devoted to the introduction of the spectral discretization, the description of some conservation properties, and the derivation of the error estimates. In section \ref{sec4}, the Galerkin semidiscretization is reformulated in collocation form and then integrated in time with the IMR. A complete description of the implementation is made and some numerical experiments illustrating the accuracy and general performance of the full discretization are shown.

We describe the main notation used throughout the paper. Let $\Omega\subset\mathbb{R}^{2}$ be a domain. The usual inner product in $L^{2}=L^{2}(\Omega)$ is denoted by $(\cdot,\cdot)$, with the associated norm given by $||\cdot||$. In the case of the product space $L^{2}\times L^{2}$, we will change the notation to $\langle\cdot,\cdot\rangle$, meaning, for ${\bf u}, {\bf v}\in L^{2}\times L^{2}$
\begin{eqnarray*}
\langle {\bf u}, {\bf v}\rangle=(u_{1},v_{1})+(u_{2},v_{2}),\quad {\bf u}=(u_{1},u_{2}), {\bf v}=(v_{1},v_{2}).
\end{eqnarray*}
Let also $|\cdot|_{\infty}$ be the norm in the Banach space $L^{\infty}=L^{\infty}(\Omega)$.

For $s\geq 0$, $H^{s}=H^{s}(\Omega)$ will denote the $L^{2}$-based Sobolev space on $\Omega$ of order $s$ (where $H^{0}:=L^{2}$) and usual norm $||\cdot||_{s}$. For the sake of simplicity, the same notation for the norms will be used when dealing with product spaces $H^{s}\times H^{s}$ and $L^{\infty}\times L^{\infty}$, with
\begin{eqnarray*}
||{\bf u}||_{s}=(||u_{1}||_{s}^{2}+||u_{2}||_{s}^{2})^{1/2},\; |{\bf u}|_{\infty}=\max\{|u_{1}|_{\infty},|u_{2}|_{\infty}\},\; {\bf u}=(u_{1},u_{2}).
\end{eqnarray*}

If $T>0$ and $k\geq 0$ is an integer, $C^{k}(0,T,H^{s}), s\geq 0$ will stand for the space of $k$th-order continuously differentiable functions $u:[0,T]\rightarrow H^{s}$.

The symbol $f\lesssim g$ stands for the inequality $f\leq C g$, for some constant $C$. The dependence of $C$ on the data and parameters of the problem will be specified in each case. 
\section{The periodic ivp for (\ref{BFD1a}), (\ref{BFD1b})}
\label{sec2}
In this section we analyze the well-posedness of the periodic ivp of the B-FD systems (\ref{BFD1a}), (\ref{BFD1b}), with period $2\pi$ in both variables for simplicity. These will be written in the form
\begin{eqnarray}
\partial_{t}\zeta+\frac{1}{\gamma}R_{b}\cdot\left((1-\zeta){\bf u}\right)-\frac{1}{\gamma^{2}}\widetilde{S}_{1}\cdot{\bf u}+\frac{1}{\gamma}\widetilde{S}_{2}(a)\cdot{\bf u}&=&0,\label{BFD21a}\\
\partial_{t}{\bf u}+(1-\gamma)\widetilde{R}_{d}(c)\zeta-\frac{1}{2\gamma}R_{d}|{\bf u}|^{2}&=&0,\label{BFD21b}
\end{eqnarray}
for $2\pi$-periodic, real functions $\zeta=\zeta(x,y,t), {\bf u}=(u_{1}(x,y,t),u_{2}(x,y,t)), 0\leq t\leq T, (x,y)\in\Omega=(0,2\pi)^{2}$, and $2\pi$-periodic, real functions as initial conditions
\begin{eqnarray}
\zeta(x,y,0)= \zeta_{0}(x,y),\quad {\bf u}(x,y,0)={\bf u}^{0}(x,y)=(u_{1}^{0}(x,y),u_{2}^{0}(x,y)).\label{BFD21c}
\end{eqnarray}
In (\ref{BFD21a}), (\ref{BFD21b}) some nonlocal operators acting on the Sobolev spaces $H^{s}, s\in\mathbb{R}$, are introduced. They are described in Table \ref{BFD_t2}, along with their corresponding Fourier symbol, \cite{DDS3}, where ${\bf k}=(k_{x},k_{y})$ and $|{\bf k}|=\sqrt{k_{x}^{2}+k_{y}^{2}}$.
\begin{table}[ht]
\centering
\begin{tabular}{|c|c|}
    \hline
Operator&Fourier symbol\\\hline\hline
$T_{b}=(Id-b\Delta)^{-1}:H^{s-2}\rightarrow H^{s}$&$1/(1+b|{\bf k}|^{2})$\\\hline
$T_{d}=(Id-d\Delta)^{-1}:H^{s-2}\rightarrow H^{s}$&$1/(1+d|{\bf k}|^{2})$\\\hline
$T_{c}=(Id+c\Delta):H^{s}\rightarrow H^{s-2}$&$1-c|{\bf k}|^{2}$\\\hline
$S_{1}=|D|{\rm coth}{|D|}:H^{s}\rightarrow H^{s-1}$&$\tau_{1}({\bf k})=|{\bf k}|{\rm coth}{|{\bf k}|}$\\\hline
$R_{b}=T_{b}\nabla$&$\left(\frac{ik_{x}}{1+b|{\bf k}|^{2}},\frac{ik_{y}}{1+b|{\bf k}|^{2}}\right)$\\\hline
$R_{d}=T_{d}\nabla$&$\left(\frac{ik_{x}}{1+d|{\bf k}|^{2}},\frac{ik_{y}}{1+d|{\bf k}|^{2}}\right)$\\\hline
$\widetilde{R}_{d}(c)=T_{d}T_{c}\nabla$&$\left(\frac{ik_{x}(1-c|{\bf k}|^{2})}{1+d|{\bf k}|^{2}},\frac{ik_{y}(1-c|{\bf k}|^{2})}{1+d|{\bf k}|^{2}}\right)$\\\hline
$\widetilde{S}_{1}=T_{b}S_{1}\nabla$&$\left(\frac{ik_{x}\tau_{1}({\bf k})}{1+b|{\bf k}|^{2}},\frac{ik_{y}\tau_{1}({\bf k})}{1+b|{\bf k}|^{2}}\right)$\\\hline
$\widetilde{S}_{2}(a)=T_{b}(a\Delta+\frac{1}{\gamma^{2}}S_{1}^{2})\nabla$&$\left(\frac{ik_{x}(-a|{\bf k}|^{2}+\frac{1}{\gamma^{2}}\tau_{1}({\bf k})^{2})}{1+b|{\bf k}|^{2}},\frac{ik_{y}(-a|{\bf k}|^{2}+\frac{1}{\gamma^{2}}\tau_{1}({\bf k})^{2})}{1+b|{\bf k}|^{2}}\right)$\\\hline
\hline
\end{tabular}
\caption{Operators defined in (\ref{BFD21a}), (\ref{BFD21b}) and corresponding Fourier symbols.}
\label{BFD_t2}
\end{table}

Some properties of the operators in Table \ref{BFD_t2}, derived from their Fourier representation, will be used in the sequel. As in the 1D case, the operators $T_{k}, k=b,d>0$, satisfy
\begin{eqnarray}
||T_{k}f||_{s}\lesssim||f||_{s-2},\quad s\in\mathbb{R},\quad f\in H^{s-2},\label{BFD22a}
\end{eqnarray}
while, for $c\leq 0$
\begin{eqnarray}
||T_{c}f||_{s-2}\lesssim||f||_{s},\quad s\in\mathbb{R},\quad f\in H^{s}.\label{BFD22b}
\end{eqnarray}
The symbol of $S_{1}$ behaves like $|{\bf k}|$ for $|{\bf k}|\geq 1$ and, from Taylor expansions, like $1+\frac{|{\bf k}|^{2}}{3}+O(|{\bf k}|^{4})$ for $|{\bf k}|$ small.This implies, when $a\leq 0$, that
\begin{eqnarray}
||\widetilde{S}_{1}f||_{s}&\lesssim&||f||_{s},\quad s\in\mathbb{R},\quad f\in H^{s}.\label{BFD22c}\\
||\widetilde{S}_{2}(a)f||_{s}&\lesssim&||f||_{s+1},\quad s\in\mathbb{R},\quad f\in H^{s+1}.\label{BFD22d}
\end{eqnarray}
Finally, for $k=b,d$
\begin{eqnarray}
||R_{k}f||_{s}&\lesssim&||f||_{s-1},\quad s\in\mathbb{R},\quad f\in H^{s-1},\label{BFD22e}
\end{eqnarray}
while
\begin{eqnarray}
||\widetilde{R}_{d}(c)f||_{s}&\lesssim&||f||_{s+1},\quad s\in\mathbb{R},\quad f\in H^{s+1},c<0,\label{BFD22f}\\
||\widetilde{R}_{d}(0)f||_{s}&\lesssim&||f||_{s-1},\quad s\in\mathbb{R},\quad f\in H^{s-1}.\label{BFD22g}
\end{eqnarray}
From Table \ref{BFD_t2} and (\ref{BFD21a}), (\ref{BFD21b}),, we can also derive the system for the Fourier coefficients of the variables $\zeta, {\bf u}$. This has the form
\begin{eqnarray}
\widehat{\zeta}_{t}+\frac{1}{\gamma}\left(\frac{ik_{x}}{1+b|{\bf k}|^{2}}\widehat{u_{1}}+\frac{ik_{y}}{1+b|{\bf k}|^{2}}\widehat{u_{2}}\right)-\frac{1}{\gamma}\left(\frac{ik_{x}}{1+b|{\bf k}|^{2}}\widehat{\zeta u_{1}}+\frac{ik_{y}}{1+b|{\bf k}|^{2}}\widehat{\zeta u_{2}}\right)&&\nonumber\\
-\frac{1}{\gamma^{2}}\left(\frac{ik_{x}\tau_{1}({\bf k})}{1+b|{\bf k}|^{2}}\widehat{u_{1}}+\frac{ik_{y}\tau_{1}({\bf k})}{1+b|{\bf k}|^{2}}\widehat{u_{2}}\right)&&\nonumber\\
+\frac{1}{\gamma}\left(\frac{ik_{x}(-a|{\bf k}|^{2}+\frac{1}{\gamma^{2}}\tau_{1}({\bf k})^{2})}{1+b|{\bf k}|^{2}}\widehat{u_{1}}+\frac{ik_{y}(-a|{\bf k}|^{2}+\frac{1}{\gamma^{2}}\tau_{1}({\bf k})^{2})}{1+b|{\bf k}|^{2}}\widehat{u_{2}}\right)=0,&&\label{BFD23a}\\
(\widehat{u_{1}})_{t}+(1-\gamma)\frac{ik_{x}(1-c|{\bf k}|^{2})}{1+d|{\bf k}|^{2}}\widehat{\zeta}-\frac{1}{2\gamma}\frac{ik_{x}}{1+d|{\bf k}|^{2}}\widehat{(u_{1}^{2}+u_{2}^{2})}=0,&&\label{BFD23b}\\
(\widehat{u_{2}})_{t}+(1-\gamma)\frac{ik_{y}(1-c|{\bf k}|^{2})}{1+d|{\bf k}|^{2}}\widehat{\zeta}-\frac{1}{2\gamma}\frac{ik_{y}}{1+d|{\bf k}|^{2}}\widehat{(u_{1}^{2}+u_{2}^{2})}=0,&&\label{BFD23c}
\end{eqnarray}
where for $f=\zeta, u_{1}, u_{2}$, $\widehat{f}=\widehat{f}({\bf k},t), {\bf k}=(k_{x},k_{y})\in\mathbb{Z}^{2}$ denotes the $k$th Fourier coefficient of $f(\cdot,t), t\geq 0$. The initial conditions for (\ref{BFD23a})-(\ref{BFD23c}) are
\begin{eqnarray}
\widehat{\zeta}({\bf k},0)=\widehat{\zeta_{0}}({\bf k}),\quad \widehat{{\bf u}}({\bf k},0)=\widehat{{\bf u}^{0}}({\bf k}),\quad {\bf k}\in\mathbb{Z}^{2}.\label{BFD23d}
\end{eqnarray}
In matrix form,  (\ref{BFD23a})-(\ref{BFD23c}) can be written as
\begin{eqnarray}
\frac{d}{dt}\begin{pmatrix}\widehat{\zeta}\\\widehat{u_{1}}\\\widehat{u_{2}}\end{pmatrix}({\bf k},t)+i|{\bf k}|\mathcal{A}({\bf k})\begin{pmatrix}\widehat{\zeta}\\\widehat{u_{1}}\\\widehat{u_{2}}\end{pmatrix}({\bf k},t)+i\mathcal{F}(\widehat{\zeta},\widehat{u_{1}},\widehat{u_{2}})({\bf k},t)=0,\label{BFD24a}
\end{eqnarray}
where 
\begin{eqnarray}
\mathcal{A}({\bf k})&=&\begin{pmatrix}0&\frac{k_{x}}{|{\bf k}|}\frac{g({\bf k})}{1+b|{\bf k}|^{2}}&\frac{k_{y}}{|{\bf k}|}\frac{g({\bf k})}{1+b|{\bf k}|^{2}}\\
(1-\gamma)\frac{k_{x}}{|{\bf k}|}\frac{(1-c|{\bf k}|^{2})}{1+d|{\bf k}|^{2}}&0&0\\
(1-\gamma)\frac{k_{y}}{|{\bf k}|}\frac{(1-c|{\bf k}|^{2})}{1+d|{\bf k}|^{2}}&0&0\end{pmatrix},\label{BFD24b}\\
\mathcal{F}(\widehat{\zeta},\widehat{{\bf u}})({\bf k},t)&=&\begin{pmatrix}\frac{1}{\gamma}\frac{k_{x}\widehat{(\zeta u_{1})}({\bf k},t)+k_{y}\widehat{(\zeta u_{2})}({\bf k},t)}{1+b|{\bf k}|^{2}}\\
\frac{1}{2\gamma}\frac{k_{x}}{1+d|{\bf k}|^{2}}\widehat{|{\bf u}|^{2}}({\bf k},t)\\
\frac{1}{2\gamma}\frac{k_{y}}{1+d|{\bf k}|^{2}}\widehat{|{\bf u}|^{2}}({\bf k},t)\end{pmatrix},\label{BFD24c}
\end{eqnarray}
and $g({\bf k})$ is defined in (\ref{BFD1d}). For the study of well-posedness of (\ref{BFD21a})-(\ref{BFD21c}), we will extend Theorem 2.1 of \cite{A} to the periodic case. The extension is direct, since the arguments for the proof make also use of the Contraction Mapping Theorem, from the Fourier representation of the solutions, based in this case on Fourier series, and (\ref{BFD23a})-(\ref{BFD23d}) or, alternatively,  (\ref{BFD24a})-(\ref{BFD24c}).
\begin{theorem}
\label{TH21}
Let $s>0$.
\begin{itemize}
\item[(i)] Assume that $b,d>0, a\leq 0, c<0$. Let $(\zeta_{0},{\bf u}^{0})\in H^{s}\times H^{s}\times H^{s}$. Then there exist $T>0$ and a unique solution $(\zeta,{\bf u})\in X_{T}:=C(0,T,H^{s})^{3}$ of (\ref{BFD21a})-(\ref{BFD21c}).
\item[(ii)] Assume that $b,d>0, a\leq 0, c=0$. Let $(\zeta_{0},{\bf u}^{0})\in H^{s-1}\times H^{s}\times H^{s}$. Then there exist $T>0$ and a unique solution $(\zeta,{\bf u})\in X_{T}:=C(0,T,H^{s-1})\times C(0,T,H^{s})^{2}$ of (\ref{BFD21a})-(\ref{BFD21c}).
\end{itemize}
\end{theorem}
\begin{remark}
As in the case of the ivp, when $b=d$, the system (\ref{BFD21a}), (\ref{BFD21b}) admits a Hamiltonian formulation, \cite{AnguloS2019}
\begin{eqnarray}
\partial_{t}\begin{pmatrix}\zeta\\{\bf u}\end{pmatrix}=\mathcal{J}\delta \mathcal{H}(\zeta,{\bf u}),\nonumber
\end{eqnarray}
where $\delta$ denotes variational derivative w.r.t. $(\zeta,{\bf u})$ and
\begin{eqnarray}
\mathcal{J}&=&-T_{b}\begin{pmatrix}0&\nabla\cdot\\\nabla&0\end{pmatrix},\nonumber\\
\mathcal{H}(\zeta,{\bf u})&=&\frac{1}{2}\int_{\Omega}\left((1-\gamma)\zeta^{2}+\frac{1}{\gamma}|{\bf u}|^{2}-\frac{1}{\gamma}\zeta|{\bf u}|^{2}-c(1-\gamma)|\nabla\zeta|^{2}\right.\nonumber\\
&&\left.-\frac{a}{\gamma}|\nabla{\bf u}|^{2}-\frac{1}{\gamma^{2}}\langle {\bf u},L_{1}{\bf u}\rangle-\frac{1}{\gamma^{3}}|L_{1}{\bf u}|^{2}\right)d{\bf x},\label{BFD25}
\end{eqnarray}
with
\begin{eqnarray*}
|\nabla {\bf u}|^{2}=|\nabla u_{1}|^{2}+|\nabla u_{2}|^{2},\quad
L_{1}{\bf u}=(S_{1}u_{1},S_{1}u_{2}).
\end{eqnarray*}
\end{remark}
\section{Spectral semidiscretization}
\label{sec3}
For an integer $N\geq 1$, we consider the finite dimensional space $S_{N}$ of trigonometric polynomials
\begin{eqnarray*}
S_{N}={\rm span}\{\phi_{{\bf k}}(x,y), {\bf k}=(k_{x},k_{y}), k_{j}\in\mathbb{Z}, -N\leq k_{j}\leq N, j=x,y\},
\end{eqnarray*}
where
\begin{eqnarray*}
\phi_{\bf k}(x,y)=e^{i{\bf x}\cdot{\bf k}},\quad {\bf x}=(x,y), {\bf x}\cdot{\bf k}=k_{x}x+k_{y}y.
\end{eqnarray*}
The space $S_{N}$ satisfies the following inverse inequalities, \cite{CHQZ}: {For integers $0\leq l\leq m$, $\phi\in S_{N}$
\begin{eqnarray}
||\phi||_{H^{m}}&\lesssim &N^{m-l}||\phi||_{H^{l}},\label{BFD31a}\\
||\phi||_{m,\infty}&\lesssim &N^{m+1-l}||\phi||_{H^{l}}.\label{BFD31b}
\end{eqnarray}
}
We define the Fourier-Galerkin approximation of (\ref{BFD23a})-(\ref{BFD23c}) in $S_{N}$ as mappings $\zeta^{N}, u_{j}^{N}:[0,T]\rightarrow S_{N}, j=1,2,$ such that if ${\bf u}^{N}=(u_{1}^{N},u_{2}^{N})$, for $0\leq t\leq T$ 
\begin{eqnarray}
\partial_{t}\zeta^{N}+\frac{1}{\gamma}R_{b}\cdot{\bf u}^{N}-\frac{1}{\gamma}R_{b}\cdot\left(\widetilde{P}_{N}(\zeta^{N}{\bf u}^{N})\right)-\frac{1}{\gamma^{2}}\widetilde{S}_{1}\cdot{\bf u}^{N}+\frac{1}{\gamma}\widetilde{S}_{2}(a)\cdot{\bf u}^{N}=0,&&\label{BFD32a}\\
\partial_{t}{\bf u}^{N}+(1-\gamma)\widetilde{R}_{d}(c)\zeta^{N}-\frac{1}{2\gamma}R_{d}P_{N}\left(|{\bf u}^{N}|^{2}\right)=0,&&\label{BFD32b}\\
\zeta^{N}(x,y,0)= P_{N}\zeta_{0}(x,y),\quad {\bf u}^{N}(x,y,0)=\widetilde{P}_{N}{\bf u}^{0}(x,y),&&\label{BFD32c}
\end{eqnarray}
where $P_{N}$ denotes the $L^{2}$-projection onto $S_{N}$, with $\widetilde{P}_{N}{\bf u}=(P_{N}u_{1},P_{N}u_{2})$ if ${\bf u}=(u_{1},u_{2})$. Compared to the 1D case, the following projection errors are valid in $\Omega$, \cite{CHQZ}: { For integers $0\leq j\leq \mu$ and $u\in H^{\mu}(\Omega)$
\begin{eqnarray}
||u-P_{N}u||_{H^{j}}&\lesssim &N^{j-\mu}||u||_{H^{\mu}},\label{BFD33a}\\
|u-P_{N}u|_{\infty}&\lesssim & N^{1-\mu}||u||_{H^{\mu}},\quad \mu> 1.\label{BFD33b}
\end{eqnarray}
}
Note also that $P_{N}$ commutes with $R_{k},k=b,d$ and with $\partial_{x},\partial_{y}$. Using standard ode theory and from the Fourier representation of (\ref{BFD32a}), (\ref{BFD32b}) (the system for the discrete Fourier coefficients of $\zeta^{N}, u_{j}^{N}, j=1,2$), the quadratic nonlinearities, (\ref{BFD22e}), and  the inverse inequalities (\ref{BFD31a}), (\ref{BFD31b}), we can derive existence of solution of  (\ref{BFD32a})-(\ref{BFD32c}), locally in time.
\subsection{Conservation properties}
Assume $b=d$ and that the solution of (\ref{BFD32a})-(\ref{BFD32c}) exists. We can write (\ref{BFD32a}), (\ref{BFD32b}) as
\begin{eqnarray}
T_{b}^{-1}\partial_{t}\zeta^{N}&=&-\nabla\cdot A_{N}=-\nabla\cdot\left(\frac{1}{\gamma}\left(\widetilde{P}_{N}(\zeta^{N}{\bf u}^{N})\right)+\mathcal{G}{\bf u}^{N}\right),\label{BFD34a}\\
T_{b}^{-1}\partial_{t}{\bf u}^{N}&=&-\nabla B_{N}=-\nabla\left((1-\gamma)T_{c}\zeta^{N}-\frac{1}{2\gamma}P_{N}\left(|{\bf u}^{N}|^{2}\right)\right),\label{BFD34b}
\end{eqnarray}
where if ${\bf u}=(u_{1},u_{2})$
\begin{eqnarray*}
\widehat{\mathcal{G}{\bf u}}({\bf k})=g({\bf k})\widehat{{\bf u}}({\bf k}),\quad {\bf k}\in\mathbb{Z}^{2},\label{BFD34c}
\end{eqnarray*}
with $g({\bf k})$ defined in (\ref{BFD1d}). Therefore, from (\ref{BFD25})
\begin{eqnarray*}
\frac{d}{dt}\mathcal{H}(\zeta^{N},{\bf u}^{N})&=&\int_{\Omega}\partial_{t}\zeta^{N}\left((1-\gamma)\zeta^{N}-\frac{1}{2\gamma}|{\bf u}^{N}|^{2}+c\Delta\zeta^{N}\right)d{\bf x}\\
&&+\int_{\Omega}\partial_{t}{\bf u}^{N}\cdot\left(-\frac{1}{\gamma}(\zeta^{N}{\bf u}^{N})+\mathcal{G}{\bf u}^{N}\right)d{\bf x}.
\end{eqnarray*}
Now, from the property
\begin{eqnarray*}
(P_{N}\varphi,\psi)=(\varphi,\psi),\quad \varphi\in L^{2},\; \psi\in S_{N},
\end{eqnarray*}
and (\ref{BFD34a}), (\ref{BFD34b}), we have, integrating by parts and using the periodic boundary conditions
\begin{eqnarray*}
\frac{d}{dt}\mathcal{H}(\zeta^{N},{\bf u}^{N})=-\int_{\Omega}\left(\left(T_{b}\nabla\cdot A_{N}\right)B_{N}+A_{N}\cdot T_{b}\nabla B_{N}\right)d{\bf x}=0.
\end{eqnarray*}
This proves the following conservation property:
\begin{proposition}
\label{Prop31}
Assume $b=d$. While the solution $(\zeta^{N},{\bf u}^{N})$ of (\ref{BFD32a})-(\ref{BFD32c}) exists, then
\begin{eqnarray*}
\mathcal{H}(\zeta^{N},{\bf u}^{N})(\cdot,t)=\mathcal{H}(\zeta^{N},{\bf u}^{N})(\cdot,0),\quad t\geq 0,
\end{eqnarray*}
where $\mathcal{H}$ is given by (\ref{BFD25}).
\end{proposition}
\begin{remark}
In section \ref{sec4}, the implementation of (\ref{BFD34a}), (\ref{BFD34b}) in collocation form will lead, when $b=d$, to a Hamiltonian structure with a discrete version of  (\ref{BFD25}) as Hamiltonian of the semidiscrete system.
\end{remark}
\subsection{Error estimates}
We now study the error between the semidiscrete solution of (\ref{BFD32a})-(\ref{BFD32c}) and the solution of (\ref{BFD21a})-(\ref{BFD21c}) for the B-FD systems (G1)-(G3), presented in the Introduction and concerned in Theorem \ref{TH21}. The corresponding estimates will be compared with the results obtained in \cite{DDS3} for the 1D case.

We decompose, as usual, the errors in the form $\zeta^{N}-\zeta=\theta+\rho$, ${\bf u}^{N}-{\bf u}={\bf \xi}+{\bf \sigma}$, where
$\theta=\zeta^{N}-P_{N}\zeta, \rho=P_{N}\zeta-\zeta$, and ${\bf \xi}=(\xi_{1},\xi_{2})={\bf u}^{N}-\widetilde{P_{N}}{\bf u}, {\bf \sigma}=(\sigma_{1},\sigma_{2})=\widetilde{P_{N}}{\bf u}-{\bf u}$. Then, from  (\ref{BFD32a})-(\ref{BFD32c}) and (\ref{BFD21a})-(\ref{BFD21c}), we derive the following system for $\theta$ and ${\bf \xi}$:
\begin{eqnarray}
\partial_{t}\theta+\frac{1}{\gamma}R_{b}\cdot{\bf \xi}-\frac{1}{\gamma}R_{b}\cdot\left(\widetilde{P}_{N}(A)\right)-\frac{1}{\gamma^{2}}\widetilde{S}_{1}\cdot{\bf \xi}+\frac{1}{\gamma}\widetilde{S}_{2}(a)\cdot{\bf \xi}=0,&&\label{BFD35a}\\
\partial_{t}{\bf \xi}+(1-\gamma)\widetilde{R}_{d}(c)\theta-\frac{1}{2\gamma}R_{d}P_{N}\left(B\right)=0,&&\label{BFD35b}\\
\theta(x,y,0)=0,\quad {\bf \xi}(x,y,0)=0,\quad (x,y)\in\Omega,&&\label{BFD35c}
\end{eqnarray}
with
\begin{eqnarray}
A&=&\zeta^{N}{\bf u}^{N}-\zeta{\bf u}=\rho{\bf u}+\zeta{\bf \sigma}+\theta{\bf u}+\zeta\xi+\theta\sigma+\rho\xi+\rho\sigma+\theta\xi,\label{BFD35d}\\
B&=&|{\bf u}^{N}|^{2}-|{\bf u}|^{2}=2{\bf u}\cdot\sigma+2{\bf u}\cdot\xi+2\sigma\cdot\xi+{|\xi|^{2}+|\sigma|^{2}},\label{BFD35e}
\end{eqnarray}
We will also make use of the Fourier representation of (\ref{BFD34a}), (\ref{BFD34b}), which has the form
\begin{eqnarray}
&&\frac{d}{dt}\begin{pmatrix}\widehat{\theta}\\\widehat{\xi_{1}}\\\widehat{\xi_{2}}\end{pmatrix}({\bf k},t)+i|{\bf k}|\mathcal{A}({\bf k})\begin{pmatrix}\widehat{\theta}\\\widehat{\xi_{1}}\\\widehat{\xi_{2}}\end{pmatrix}({\bf k},t)+i\mathcal{F}_{N}(\widehat{\theta},\widehat{\xi_{1}},\widehat{\xi_{2}})({\bf k},t)=0,\label{BFD36a}\\
&&\widehat{\theta}({\bf k},0)=0,\quad \widehat{\xi_{j}}({\bf k},0)=0, j=1,2,\label{BFD36b}
\end{eqnarray}
where $\mathcal{A}$ is given by (\ref{BFD24b}) and
\begin{eqnarray*}
\mathcal{F}_{N}(\widehat{\theta},\widehat{\xi})=\begin{pmatrix}\frac{1}{\gamma}
\frac{{\bf k}\cdot \widehat{P_{N}A}}{1+b|{\bf k}|^{2}}\\
\frac{1}{2\gamma}\frac{k_{x}\widehat{P_{N}B}}{1+d|{\bf k}|^{2}}\\
\frac{1}{2\gamma}\frac{k_{y}\widehat{P_{N}B}}{1+d|{\bf k}|^{2}})\end{pmatrix}.\label{BFD36c}
\end{eqnarray*}
\subsubsection{Group (G1): The BBM-BBM B-FD case ($b,d>0, a=c=0$)}
We consider the first case of weakly disperrsive B-FD systems.
\begin{proposition}
\label{Prop32}
Assume that $b,d>0, a=c=0$ and that the solution $(\zeta,{\bf u})$ of (\ref{BFD21a})-(\ref{BFD21c}) satisfies $(\zeta,{\bf u})\in H^{\mu}\times H^{\mu}\times H^{\mu}, \mu>1$ for $0\leq t\leq T$. Then, for $N$ large enough
\begin{eqnarray}
\max_{0\leq t\leq T}\left(||\zeta^{N}-\zeta||+||{\bf u}^{N}-{\bf u}||_{1}\right)\leq C N^{-\mu},\label{BFD37}
\end{eqnarray}
where $C$ is a constant dependent on $\zeta_{0}, {\bf u}^{0}$ and $T$, but independent of $N$
\end{proposition}
\begin{proof}
In this case the system (\ref{BFD35a}), (\ref{BFD35b}) has the form
\begin{eqnarray}
\partial_{t}\theta+\frac{1}{\gamma}R_{b}\cdot{\bf \xi}-\frac{1}{\gamma}R_{b}\cdot\left(\widetilde{P}_{N}(A)\right)-\frac{1}{\gamma^{2}}\widetilde{S}_{1}\cdot{\bf \xi}+\frac{1}{\gamma}\widetilde{S}_{2}(0)\cdot{\bf \xi}=0,&&\label{BFD37a}\\
\partial_{t}{\bf \xi}+(1-\gamma)\widetilde{R}_{d}(0)\theta-\frac{1}{2\gamma}R_{d}P_{N}\left(B\right)=0.&&\label{BFD37b}
\end{eqnarray}
While the semidiscrete solution exists, we take the $L^{2}$ norm of (\ref{BFD37a}) and the $H^{1}\times H^{1}$ norm of (\ref{BFD37b}) to have
\begin{eqnarray}
||\partial_{t}\theta||&\leq & C\left(||R_{b}\cdot{\bf \xi}||+||R_{b}\cdot\widetilde{P}_{N}(A)||+||\widetilde{S}_{1}\cdot{\bf \xi}||+||\widetilde{S}_{2}(0)\cdot{\bf \xi}||\right),\label{BFD38a}\\
||\partial_{t}\xi||_{1}&\leq &C\left(||\widetilde{R}_{d}(0)\theta||_{1}+||R_{d}P_{N}\left(B\right)||_{1}\right).\label{BFD38b}
\end{eqnarray}
Applying (\ref{BFD22a})-(\ref{BFD22g}) to (\ref{BFD38a}), (\ref{BFD38b}) yields
\begin{eqnarray}
||\partial_{t}\theta||&\leq & C\left(||{\bf \xi}||_{1}+||A||\right),\label{BFD39a}\\
||\partial_{t}\xi||_{1}&\leq &C\left(||\theta||+||B||\right),\label{BFD39b}
\end{eqnarray}
for some constant $C$. We now estimate $||A||$ and $||B||$ from (\ref{BFD35d}), (\ref{BFD35e}). First we have
\begin{eqnarray}
||A||&\lesssim&|{\bf u}|_{\infty}||\rho||+|\zeta|_{\infty}||\sigma||+|\zeta|_{\infty}||\xi||\nonumber\\
&&+|\sigma|_{\infty}||\theta||+|\rho|_{\infty}||\xi||+|\rho|_{\infty}||\sigma||\nonumber\\
&&+|\xi|_{\infty}||\theta||+|{\bf u}|_{\infty}||\theta||.\label{BFD310}
\end{eqnarray}
Using continuity and since $\xi(0)=0$, we can consider some $t_{N}$ such that $0<t_{N}\leq T$ is the maximal time for which 
\begin{eqnarray}
|\xi|_{\infty}\leq 1,\quad 0\leq t\leq t_{N}.\label{BFD311}
\end{eqnarray}
Then, from (\ref{BFD310}),  (\ref{BFD311}), the estimates (\ref{BFD33a}), (\ref{BFD33b}), and using Theorem \ref{TH21} and  the embedding $H^{\mu}\hookrightarrow L^{\infty}$ for $\mu>1$, we have, since $N\geq 1$
\begin{eqnarray}
||A||&\lesssim&N^{-\mu}||\zeta||_{\mu}+N^{-\mu}||{\bf u}||_{\mu}+||\xi||_{1}\nonumber\\
&&+N^{1-\mu}||\theta||+N^{1-\mu}N^{-\mu}||{\bf u}||_{\mu}+||\theta||+N^{1-\mu}||\xi||_{1}\nonumber\\
&\lesssim&N^{1-2\mu}+||\xi||_{1}+||\theta||\leq N^{-\mu}+||\xi||_{1}+||\theta||.\label{BFD312}
\end{eqnarray}
Similarly, in the case of the terms in $B$ we have
\begin{eqnarray*}
||B||\lesssim|{\bf u}|_{\infty}(||\sigma||+||\xi||)+|\sigma|_{\infty}||\xi||
+|\sigma|_{\infty}||\sigma||+|\xi|_{\infty}||\xi||.
\end{eqnarray*}
Then, from (\ref{BFD311}) and the estimates (\ref{BFD33a}), (\ref{BFD33b}), since $\mu>1$
\begin{eqnarray}
||B||&\lesssim&N^{-\mu}+||\xi||_{1}+N^{1-\mu}||\xi||_{1}+N^{1-\mu}N^{-\mu}\lesssim N^{-\mu}+||\xi||_{1}.\label{BFD313}
\end{eqnarray}
Therefore, applying (\ref{BFD312}), (\ref{BFD313}) to (\ref{BFD39a}), (\ref{BFD39b}), we have, for $\mu>1$ and $0<t\leq t_{N}$
\begin{eqnarray}
||\partial_{t}\theta(t)||+||\partial_{t}\xi(t)||_{1}\leq C\left(N^{-\mu}+||\theta(t)||+||\xi(t)||_{1}\right),\label{BFD314}
\end{eqnarray}
with $C$ independent of $N$ and $t_{N}$. Using (\ref{BFD35c}) and Gronwall's inequality, it holds that
\begin{eqnarray}
||\theta(t)||+||\xi(t)||_{1}\leq CN^{-\mu},\quad 0<t\leq t_{N}.\label{BFD315}
\end{eqnarray}
Note that from the inverse inequalities (\ref{BFD31a}), (\ref{BFD31b})
\begin{eqnarray}
|\xi|_{\infty}\lesssim N||\xi||_{1}.\label{BFD315a}
\end{eqnarray}
Then, using (\ref{BFD315}) and since $\mu>1$
\begin{eqnarray*}
|\xi|_{\infty}\lesssim N^{1-\mu},\label{BFD315b}
\end{eqnarray*}
which implies that $t_{N}$ is not maximal in (\ref{BFD311}) if we take $N$ large enough. Continuing this argument up to $t_{N}=T$ we obtain (\ref{BFD315}) for $0<t\leq T$, leading to (\ref{BFD37}).
\end{proof}
\begin{remark}
\label{rem1}
Note that the final argument is different from that in the 1D case; cf. \cite{DDS3}. In addtion, a sharper estimate, instead of (\ref{BFD315a}), is
\begin{eqnarray}
|\xi|_{\infty}\lesssim \sqrt{\log{N}}||\xi||_{1}.\label{esti1}
\end{eqnarray}
On the other hand, the estimate (\ref{BFD37}) is optimal for $\zeta$ and ${\bf u}$ in the $L^{2}$ norm.
\end{remark}
\subsubsection{Group (G2): The generic B-FD case ($b,d>0, a, c<0$)}
We now consider the generic system (\ref{BFD35a}), (\ref{BFD35b})  in its Fourier representation  (\ref{BFD36a}), (\ref{BFD36b}). Let $g$ be given by (\ref{BFD1d}).
Note first that, using the properties mentioned in section \ref{sec2}, when $|{\bf k}|$ is small
\begin{eqnarray*}
g({\bf k})\approx \frac{1}{\gamma}\left(1+\frac{1}{\gamma}\left(\frac{1}{\gamma}-1\right)+|a||{\bf k}|^{2}+\frac{|{\bf k}|^{2}}{3\gamma}\left(\frac{2}{\gamma}-1\right)\right)+O(|{\bf k}|^{4}),
\end{eqnarray*}
which, for $|{\bf k}|$ small enough, behaves like $d_{1}'+d_{2}'|{\bf k}|^{2}$ for some constants $d_{1}',d_{2}'>0$.
When $|{\bf k}|$ is large
\begin{eqnarray*}
g({\bf k})\approx \frac{1}{\gamma}\left(1-\frac{|{\bf k}|}{\gamma}+|a||{\bf k}|^{2}+\frac{|{\bf k}|^{2}}{\gamma^{2}}\right),
\end{eqnarray*}
which, for $|{\bf k}|$ large enough, behaves like $d_{1}+d_{2}|{\bf k}|^{2}$ for some constants $d_{1},d_{2}>0$. This will be used in the sequel.

We first diagonalize (\ref{BFD36a}), (\ref{BFD36b}). The eigenvalues of $\mathcal{A}({\bf k})$ are $\{0,\pm\sigma({\bf k})\}$, where
\begin{eqnarray}
\sigma({\bf k})=\left(\frac{(1-\gamma)g({\bf k})(1-c|{\bf k}|^{2})}{(1+b|{\bf k}|^{2})(1+d|{\bf k}|^{2})}\right)^{1/2},\label{sigma}
\end{eqnarray}
and eigenvectors given by the columns of
\begin{eqnarray}
P({\bf k})=\begin{pmatrix}0&\alpha({\bf k})&-\alpha({\bf k})\\-\frac{k_{y}}{|{\bf k}|}&\frac{k_{x}}{|{\bf k}|}&\frac{k_{x}}{|{\bf k}|}\\\frac{k_{x}}{|{\bf k}|}&\frac{k_{y}}{|{\bf k}|}&\frac{k_{y}}{|{\bf k}|}\end{pmatrix},
\alpha({\bf k})=\left(\frac{g({\bf k})(1+d|{\bf k}|^{2})}{(1-\gamma)(1+b|{\bf k}|^{2})(1-c|{\bf k}|^{2})}\right)^{1/2}. \label{BFD317}
\end{eqnarray}
Then
\begin{eqnarray*}
P({\bf k})^{-1}\mathcal{A}({\bf k})P({\bf k})=D({\bf k})=\begin{pmatrix}0&0&0\\0&\sigma({\bf k})&0\\0&0&-\sigma({\bf k})\end{pmatrix}.
\end{eqnarray*}
We consider the change of variables, cf. \cite{DMS2007}
\begin{eqnarray}
\begin{pmatrix}\widehat{\eta}\\\widehat{v_{1}}\\\widehat{v_{2}}\end{pmatrix}=P^{-1}({\bf k})\begin{pmatrix}\widehat{\theta}\\\widehat{\xi_{1}}\\\widehat{\xi_{2}}\end{pmatrix},\label{BFD317a}
\end{eqnarray}
that is
\begin{eqnarray*}
\widehat{\eta}({\bf k},t)&=&-\frac{k_{y}}{|{\bf k}|}\widehat{\xi_{1}}({\bf k},t)+\frac{k_{x}}{|{\bf k}|}\widehat{\xi_{2}}({\bf k},t),\label{BFD318a}\\
\widehat{v_{1}}({\bf k},t)&=&\frac{1}{2\alpha({\bf k})}\widehat{\theta}({\bf k},t)+\frac{k_{x}}{2|{\bf k}|}\widehat{\xi_{1}}({\bf k},t)+\frac{k_{y}}{2|{\bf k}|}\widehat{\xi_{2}}({\bf k},t),\label{BFD318b}\\
\widehat{v_{2}}({\bf k},t)&=&-\frac{1}{2\alpha({\bf k})}\widehat{\theta}({\bf k},t)+\frac{k_{x}}{2|{\bf k}|}\widehat{\xi_{1}}({\bf k},t)+\frac{k_{y}}{2|{\bf k}|}\widehat{\xi_{2}}({\bf k},t).\label{BFD318c}
\end{eqnarray*}
Note that, from (\ref{BFD317}), $\alpha({\bf k})\neq 0$ and, due to the properties of $g({\bf k})$ mentioned above, $\alpha({\bf k})$ has order zero. Therefore
\begin{eqnarray}
C_{1}\left(||\eta||_{s}+||{\bf v}||_{s}\right)\leq ||\theta||_{s}+||\xi||_{s}\leq C_{2}\left(||\eta||_{s}+||{\bf v}||_{s}\right),\label{BFD318d}
\end{eqnarray}
for $s\geq 0$ and some constants $C_{1},C_{2}$ independent of $N, S, \theta,\eta,{\bf v}$ and $\xi$. Then the system  (\ref{BFD36a}), (\ref{BFD36b}) is diagonalized in the form
\begin{eqnarray}
\frac{d}{dt}\begin{pmatrix}\widehat{\eta}\\\widehat{v_{1}}\\\widehat{v_{2}}\end{pmatrix}({\bf k},t)+i|{\bf k}|D({\bf k})\begin{pmatrix}\widehat{\eta}\\\widehat{v_{1}}\\\widehat{v_{2}}\end{pmatrix}(k,t)
+iP({\bf k})^{-1}\mathcal{F}_{N}(\widehat{\eta},\widehat{v_{1}},\widehat{v_{2}})=0,&&\nonumber\\
\widehat{\eta}({\bf k},0)=\widehat{v_{1}}({\bf k},0)=\widehat{v_{2}}({\bf k},0)=0,&&\nonumber
\end{eqnarray}
for ${\bf k}=(k_{x},k_{y}), -N\leq k_{x},k_{y}\leq N$ or, in physical variables
\begin{eqnarray}
&&\frac{d}{dt}\begin{pmatrix}{\eta}\\{v_{1}}\\{v_{2}}\end{pmatrix}+\mathcal{B}\begin{pmatrix}{\eta}\\{v_{1}}\\{v_{2}}\end{pmatrix}+{F}({\eta},{v_{1}},{v_{2}})=0,\label{BFD319a}\\
&&\eta(0)={\bf v}(0)=0,\label{BFD319b}
\end{eqnarray}
where $\mathcal{B}$ is the operator with Fourier symbol
$
i|{\bf k}|D({\bf k})
$ and
the Fourier representation of ${F}$ is
$iP({\bf k})^{-1}\mathcal{F}_{N}(\widehat{\eta}({\bf k},t),\widehat{v_{1}}({\bf k},t),\widehat{v_{2}}({\bf k},t)).$ 

We note that $\mathcal{B}$ is skew-adjoint with domain $X=H^{1}\times H^{1}\times H^{1}$.  Let $\mathcal{U}(t), t\geq 0$ be the semigroup whose infinitesimal generator is $-\mathcal{B}$. From  the previous analysis it holds that $\mathcal{U}(t)$ is unitary. Duhamel's formula for (\ref{BFD319a}), (\ref{BFD319b}) leads to
\begin{eqnarray*}
\begin{pmatrix}{\eta}\\{v_{1}}\\{v_{2}}\end{pmatrix}(t)
=-\int_{0}^{t}\mathcal{U}(t-\tau){F}({\eta}(\tau),{v_{1}}(\tau),{v_{2}}(\tau))d\tau.
\end{eqnarray*}
Then $||\mathcal{U}F||_{X}=||F||_{X}$ and
\begin{eqnarray}
\left\|\begin{pmatrix}{\eta}\\{v_{1}}\\{v_{2}}\end{pmatrix}(t)\right\|_{X}\leq \int_{0}^{t}||{F}({\eta}(\tau),{v_{1}}(\tau),{v_{2}}(\tau))||_{X}d\tau.\label{BFD319c}
\end{eqnarray}
We now estimate the right-hand side of (\ref{BFD319c}). Since $\alpha$ has order zero and using Grisvard's lemma, \cite{Grisvard}
\begin{eqnarray}
||{F}({\eta},{v_{1}},{v_{2}})||_{X}&\lesssim&||\mathcal{F}({\eta},{v_{1}},{v_{2}})||_{X}\lesssim ||(Id-b\Delta)^{-1}\nabla\cdot \widetilde{P}_{N}(A)||_{1}\nonumber\\
&&+||(Id-d\Delta)^{-1}\nabla P_{N}(B)||_{1}\lesssim||A||+||B||.\label{BFD319d}
\end{eqnarray}
We now estimate these two terms. Let $0<t_{N}\leq T$ be the maximal time for which (\ref{BFD311}) holds. Using Grisvard's lemma again, we have
\begin{eqnarray*}
||A||&\lesssim&||{\bf u}||_{1}||\rho||_{1}+||\zeta||_{1}||\sigma||_{1}+||\zeta||_{1}||\xi||_{1}+||\sigma||_{1}||\theta||_{1}+||\rho||_{1}||\xi||_{1}+||\rho||_{1}||\sigma||_{1}\nonumber\\
&&+|\xi|_{\infty}||\theta||_{1}+||{\bf u}||_{1}||\theta||_{1}.
\end{eqnarray*}
Then, from (\ref{BFD310}),  (\ref{BFD311}), and the estimates (\ref{BFD33a}), (\ref{BFD33b}), and using Theorem \ref{TH21} we have, since $\mu>1, N\geq 1$
\begin{eqnarray}
||A||&\lesssim&N^{1-\mu}||\zeta||_{\mu}+N^{1-\mu}||{\bf u}||_{\mu}+||\xi||_{1}\nonumber\\
&&+N^{1-\mu}||\theta||_{1}+N^{1-\mu}N^{-\mu}||{\bf u}||_{\mu}+||\theta||_{1}+N^{1-\mu}||\xi||_{1}\nonumber\\
&\lesssim&N^{1-2\mu}+||\xi||_{1}+||\theta||_{1}\leq N^{-\mu}+||\xi||_{1}+||\theta||_{1}.\label{BFD320a}
\end{eqnarray}
In the case of $B$, by components and using (\ref{BFD311}) we have
\begin{eqnarray*}
||B||\lesssim||{\bf u}||_{1}(||\sigma||_{1}+||\xi||_{1})+||\sigma||_{1}||\xi||_{1}+||\sigma||_{1}^{2}+||\xi||_{1}.
\end{eqnarray*}
Then, from (\ref{BFD311}) and the estimates (\ref{BFD33a}), (\ref{BFD33b})
\begin{eqnarray}
||B||&\lesssim&N^{1-\mu}+||\xi||_{1}+N^{1-\mu}||\xi||_{1}+N^{2-2\mu}\lesssim N^{1-\mu}+||\xi||_{1}.\label{BFD320b}
\end{eqnarray}
Therefore, from (\ref{BFD318d}), (\ref{BFD319c}), (\ref{BFD319d}), (\ref{BFD320a}), (\ref{BFD320b}), and for $0<t\leq t_{N}$
\begin{eqnarray*}
||\theta||_{1}+||\xi||_{1}&\leq&C\int_{0}^{t}(||A||+||B||)d\tau\leq C\int_{0}^{t}(N^{1-\mu}+||\xi||_{1}+||\theta||_{1})d\tau,
\end{eqnarray*}
for some constant $C$. From Gronwall's lemma, it holds that, if $0<t\leq t_{N}$
\begin{eqnarray}
||\theta(t)||_{1}+||\xi(t)||_{1}\leq CN^{1-\mu}.\label{BFD320c}
\end{eqnarray}
Note that, from (\ref{esti1}) and (\ref{BFD320c})
\begin{eqnarray*}
|\xi|_{\infty}\lesssim \sqrt{\log{N}}||\xi||_{1}\lesssim C\sqrt{\log{N}}N^{1-\mu}.
\end{eqnarray*}
For $N\geq 1$ there holds $\log{N}\leq N$ and thus
\begin{eqnarray*}
|\xi|_{\infty}\lesssim N^{3/2-\mu},
\end{eqnarray*}
which implies that $t_{N}$ is not maximal in (\ref{BFD311}) if we take $N$ large enough and $\mu>3/2$. The same argument as that in Proposition \ref{Prop32} finally proves the following result:
\begin{proposition}
\label{Prop33}
Assume that $b,d>0, a,c<0$ and that the solution $(\zeta,{\bf u})$ of (\ref{BFD21a})-(\ref{BFD21c}) satisfies $(\zeta,{\bf u})\in H^{\mu}\times H^{\mu}\times H^{\mu}, \mu>3/2$ for $0\leq t\leq T$. Then, if $N$ is large enough
\begin{eqnarray}
\max_{0\leq t\leq T}\left(||\zeta^{N}-\zeta||_{1}+||{\bf u}^{N}-{\bf u}||_{1}\right)\leq C N^{1-\mu},\label{BFD321}
\end{eqnarray}
where $C$ is a constant dependent on $\zeta_{0}, {\bf u}^{0}$ and $T$, but independent of $N$
\end{proposition}
Note that the estimate (\ref{BFD321}) is optimal for $\zeta$ and ${\bf u}$ in the $H^{1}$ norm.
\subsubsection{Group (G3): Other B-FD systems}
We briefly discuss other B-FD systems in group (G3).
\begin{itemize}
\item[(i)] $b,d>0, a=0, c<0$. In this case, $\alpha({\bf k})$ is still of order zero and then (\ref{BFD318d}) holds. So we can use the same proof of Proposition \ref{Prop33} to have (\ref{BFD321}).
\item[(ii)] $b,d>0, a<0, c=0$. In this case, $\alpha({\bf k})$ is of order one with $\alpha({\bf k})\neq 0, {\bf k}\in\mathbb{R}^{2}$. Then $(\theta,\xi)\in H^{s-1}\times H^{s}\times H^{s}$ iff $(\eta,{\bf v})\in H^{s}\times H^{s}\times H^{s}$, and (\ref{BFD318d}) changes to
\begin{eqnarray*}
C_{1}\left(||\eta||_{s}+||{\bf v}||_{s}\right)\leq ||\theta||_{s-1}+||\xi||_{s}\leq C_{2}\left(||\eta||_{s}+||{\bf v}||_{s}\right).
\end{eqnarray*}
Then we can apply the proof of Proposition \ref{Prop33} and the change (\ref{BFD317a}) leads, instead of (\ref{BFD321}), to the estimate
\begin{eqnarray*}
\max_{0\leq t\leq T}\left(||\zeta^{N}-\zeta||+||{\bf u}^{N}-{\bf u}||_{1}\right)\leq C N^{1-\mu},
\end{eqnarray*}
which is not optimal for $\zeta$ in the $L^{2}$ norm.
\end{itemize}
\section{Full discretization and numerical experiments}
\label{sec4}
In this section we will introduce the fully discrete scheme to approximate (\ref{BFD21a})-(\ref{BFD21c}) and illustrate its performance with some numerical experiments.
\subsection{Full discretization}
\label{sec41}
The periodic ivp (\ref{BFD21a})-(\ref{BFD21c}) is discretized in space with a Fourier collocation method based on a uniform grid $(x_{l},y_{m})$ with
\begin{eqnarray}
x_{l}=lh,\quad y_{m}=mh,\quad l,m=0,\ldots,N-1,\label{BFD41}
\end{eqnarray}
for some integer $N>1$, with $h=2\pi/N$. Then the semidiscrete solution is defined as a map $t\mapsto (\zeta^{N}(t),{\bf u}^{N}(t))\in S_{N}\times S_{N}\times S_{N}, t\geq 0$, satisfying (\ref{BFD21a})-(\ref{BFD21c}) at the collocation points. Written in a nodal basis of $S_{N}$ based on (\ref{BFD41}), the numerical solution is represented by the matrix values
\begin{eqnarray*}
W^{N}(t)=\begin{pmatrix}Z^{N}(t)\\{\bf U}^{N}(t)\end{pmatrix}=\begin{pmatrix}\zeta^{N}(x_{j},y_{k},t)\\u_{1}^{N}(x_{j},y_{k},t)\\u_{2}^{N}(x_{j},y_{k},t)\end{pmatrix}_{j,k=0}^{N-1},
\end{eqnarray*}
that will satisfy an ode system 
of the form
\begin{eqnarray}
(I_{N}-b\Delta_{N})\frac{d}{dt}Z^{N}+\frac{1}{\gamma}\nabla_{N}\cdot{\bf U}^{N}-\frac{1}{\gamma}\nabla_{N}\cdot(Z^{N}.{\bf U}^{N})&&\nonumber\\
-\frac{1}{\gamma^{2}}\widetilde{S}_{1}^{N}\cdot{\bf U}^{N}+\frac{1}{\gamma}\widetilde{S}_{2}^{N}(a)\cdot {\bf U}^{N}&=&0,\label{BFD43a}\\
(I_{N}-d\Delta_{N})\frac{d}{dt}{\bf U}^{N}+(1-\gamma)(I_{N}+c\Delta_{N})\nabla_{N}Z^{N}-\frac{1}{2\gamma}\nabla_{N}|{\bf U}^{N}|^{2}&=&0,\label{BFD43b}
\end{eqnarray}
where if $D_{N,j}=(D_{N,j}(l,m))_{l,m=0}^{N-1}$ stands for the $N\times N$ Fourier pseudospectral differentiation matrix in the $j$ direction, $j=x,y$, we define 
\begin{eqnarray*}
\Delta_{N}=D_{N,x}^{2}+D_{N,y}^{2},&& \nabla_{N}=\begin{pmatrix}D_{N,x}\\D_{N,y}\end{pmatrix},\\
S_{1}^{N}=|D_{N}|{\rm coth}|D_{N}|,&&S_{2}^{N}={\rm coth}^{2}|D_{N}|,\; D_{N}=(-\Delta_{N})^{1/2},\\
\widetilde{S}_{1}^{N}=S_{1}^{N}\nabla_{N},&&\widetilde{S}_{2}(a)^{N}=(aI_{N}-\frac{1}{\gamma^{2}})\Delta_{N}\nabla_{N},
\end{eqnarray*}
with $S_{1}^{N}$ denoting the matrix operator represented in space by the matrix symbol $\tau_{1}(l,j), , l,j=0,\ldots,N-1$, and $\tau_{1}$ given in Table \ref{BFD_t2}. $I_{N}$ is the $N\times N$ identity matrix. The nonlinear terms are defined componentwise, denoted by the dot:  if ${\bf U}^{N}=\begin{pmatrix}U_{1}^{N}\\ U_{2}^{N}\end{pmatrix}$ then
\begin{eqnarray*}
Z^{N}.{\bf U}^{N}=\begin{pmatrix}Z^{N}.U_{1}^{N}\\Z^{N}.U_{2}^{N}\end{pmatrix},\;
|{\bf U}^{N}|^{2}=|U_{1}^{N}|^{2}+|U_{2}^{N}|^{2},
\end{eqnarray*}
where
\begin{eqnarray*}
|U_{l}^{N}(t)|^{2}=(|U_{l}^{N}(i,j,t)|^{2})_{i,j=0}^{N-1}, l=1,2,\; t\geq 0.
\end{eqnarray*}
Then (\ref{BFD43a}), (\ref{BFD43b}) can be written as
\begin{eqnarray}
\frac{d}{dt}W^{N}=f(W^{N}),\label{BFD42}
\end{eqnarray} 
where
\begin{eqnarray*}
f(W^{N})=-\begin{pmatrix} (I_{N}-b\Delta_{N})^{-1}\left(\frac{1}{\gamma}\nabla_{N}\cdot{\bf U}^{N}-\frac{1}{\gamma}\nabla_{N}\cdot(Z^{N}.{\bf U}^{N})-\frac{1}{\gamma^{2}}\widetilde{S}_{1}^{N}\cdot{\bf U}^{N}+\frac{1}{\gamma}\widetilde{S}_{2}^{N}(a)\cdot {\bf U}^{N}\right)\\ (I_{N}-d\Delta_{N})^{-1}\left((1-\gamma)(I_{N}+c\Delta_{N})\nabla_{N}Z^{N}-\frac{1}{2\gamma}\nabla_{N}|{\bf U}^{N}|^{2}\right)
 \end{pmatrix},\label{BFD42b}
\end{eqnarray*}
for $W=(Z,{\bf U})^{T}=(Z,U_{1},U_{2})^{T}$.
The algebraic equivalence, in the sense considered in e.~g. \cite{CHQZ}, between the collocation and Galerkin methods allows to expect similar convergence results to those of Propositions \ref{Prop32} and \ref{Prop33} for (\ref{BFD43a}), (\ref{BFD43b}).

In order to complete the full discretization, the semidiscrete system (\ref{BFD42}) is approximated at a discrete grid $t_{n}=n\Delta t, n=0,1,\ldots$, by the Implicit Midpoint Rule
\begin{eqnarray}
W^{n+1/2}=W^{n}+\frac{\Delta t}{2}f(W^{n+1/2}),\quad W^{n+1}=2W^{n+1/2}-W^{n},\label{BFD44}
\end{eqnarray}
where $W^{n}$ denotes the approximation to $W^{N}(t_{n}), n=0,1,\ldots$ The choice of (\ref{BFD44}) is determined by the well-known stability and geometric properties, which validate it as a suitable time integrator for long term simulations. Note that (\ref{BFD44}) can be written in the form
\begin{eqnarray*}
W^{n+1/2}&=&G_{n}(W^{n+1/2})=(I_{3N}+\frac{\Delta t}{2}A)^{-1}\Big(W^{n}\\
&&+\frac{\Delta t}{2}L_{N}\begin{pmatrix}0&&\nabla_{N}\cdot\\\nabla_{N}&0\end{pmatrix}\begin{pmatrix}\frac{1}{2\gamma}|{\bf U}^{n+1/2}|^{2}\\\frac{1}{\gamma}Z^{n+1/2}.{\bf U}^{n+1/2}\end{pmatrix}\Big),\\
L_{N}&=&\begin{pmatrix}(I_{N}-b\Delta_{N})^{-1}&0&0\\0&(I_{N}-d\Delta_{N})^{-1}&0\\0&0&(I_{N}-d\Delta_{N})^{-1}\end{pmatrix},
\end{eqnarray*}
where $A$ is the matrix operator with symbol $i|{\bf k}|\mathcal{A}({\bf k})$, and $\mathcal{A}$ is given by (\ref{BFD24b}). 
The full discretization is implemented fro the 2D discrete Fourier coefficients of $W^{N}$, leading to, for each ${\bf k}=(k_{x},k_{y}), 0\leq k_{x},k_{y}\leq N-1$, cf. \cite{D2026}
\begin{eqnarray}
B({\bf k})\widehat{W^{n+1/2}}({\bf k})&=&\widehat{W^{n}}({\bf k})-i\frac{\Delta t}{2}G(\widehat{W^{n+1/2}})({\bf k}),\label{BFD45a}\\
\widehat{W^{n+1}}({\bf k})&=&2\widehat{W^{n+1/2}}({\bf k})-\widehat{W^{n}}({\bf k}),\label{BFD45b}
\end{eqnarray}
where, if $W=(Z,{\bf U})^{T}=(Z,U_{1},U_{2})^{T}$,  
\begin{eqnarray*}
B({\bf k})&=&I_{3}+i|{\bf k}|\frac{\Delta t}{2}\mathcal{A}({\bf k}),\\
G(W)_{l,m}({\bf k})&=&\begin{pmatrix}\frac{1}{\gamma}\frac{k_{x}\widehat{(Z(l,m) U_{1}(l,m))}({\bf k})+k_{y}\widehat{(Z(l,m) U_{2}(l,m))}({\bf k})}{1+b|{\bf k}|^{2}}\\
\frac{1}{2\gamma}\frac{k_{x}}{1+d|{\bf k}|^{2}}\widehat{|{\bf U}(l,m)|^{2}}({\bf k})\\
\frac{1}{2\gamma}\frac{k_{y}}{1+d|{\bf k}|^{2}}\widehat{|{\bf U}(l,m)|^{2}}({\bf k})\end{pmatrix}.
\end{eqnarray*}
According to the previous analysis, the eigenvalues of $B({\bf k})$ are
\begin{eqnarray*}
\{1,1\pm i|{\bf k}|\frac{\Delta t}{2}\sigma({\bf k})\},
\end{eqnarray*}
with $\sigma({\bf k})$ given by (\ref{sigma}). The properties of $g({\bf k})$ in (\ref{BFD1d}), mentioned above, imply that $\sigma$ has order zero when $c<0$ and $-1$ when $c=0$. In any case, the operator $(I_{3N}+\frac{\Delta t}{2}A)^{-1}$ has norm $\leq 1$. Assume now that $W^{n}$ exists and $||W^{n}||_{N}<R$ for some $R>0$ and where $\|\cdot\|$ stands for the Euclidean norm in $\mathbb{R}^{3N}$. Using the quadratic character of the nonlinearities,  we have, for $W_{1}, W_{2}$ such that $||W_{j}||_{N}\leq R, j=1,2$,
\begin{eqnarray}
||G_{n}(W_{1})-G_{n}(W_{2})||_{N}&\leq& C\frac{\Delta t}{2}\left(|W_{1}|_{\infty}^{2}+|W_{2}|_{\infty}^{2}\right)^{1/2}||W_{1}-W_{2}||_{N}\nonumber\\
&\leq& C(R)\Delta t ||W_{1}-W_{2}||_{N},\label{BFD45c}
\end{eqnarray}
for some constant $C=C(R)$ depending on $R$. In addition
\begin{eqnarray}
||G_{n}(0)||_{N}=||(I_{3N}+\frac{\Delta t}{2}A)^{-1}W_{n}||_{N}\leq ||W_{n}||_{N}<R.\label{BFD45d}
\end{eqnarray}
Thus, from (\ref{BFD45c}) and (\ref{BFD45d}) we have that $G_{n}$ applies the ball of radius $R$ into itself and is contractive for $\Delta t$ small enough. This implies the existence of a solution $W^{*}=W^{n+1/2}$ of (\ref{BFD44}) with $||W^{*}||_{N}\leq R$.

 On the other hand, the system in (\ref{BFD44}) is iteratively solved by using the classical fixed point algorithm, implemented from (\ref{BFD45a}), (\ref{BFD45b}) in the $\nu$th iteration 
$\nu\mapsto \nu+1$  as
\begin{eqnarray}
\widehat{W^{[0]}}({\bf k})&=&\widehat{W^{n}}({\bf k}),\nonumber\\
B({\bf k})\widehat{W^{[\nu+1]}}({\bf k})&=&\widehat{W^{n}}({\bf k})-i\frac{\Delta t}{2}G(\widehat{W^{[\nu]}})({\bf k}),\label{BFD45e}
\end{eqnarray}
for each ${\bf k}=(k_{x},k_{y}), 0\leq k_{x},k_{y}\leq N-1$. Similar arguments can be applied to prove the convergence of the iteration (\ref{BFD45e}) when $\Delta t$ is sufficiently small and $W^{n}$ is known and bounded in the Euclidean norm.
\subsection{Numerical experiments}
In order to check the performance of the numerical scheme, two examples are considered in this section. In the first one, we approximate the periodic ivp with $\gamma=0.5, a=c=-1/6, b=1/6, d=1/2$ for which
\begin{eqnarray}
\zeta(x,y,t)&=&e^{t}\sin(x)\cos(y),\nonumber\\
u_{1}(x,y,t)&=&e^{t}\cos(x)\sin(y),\nonumber\\
u_{2}(x,y,t)&=&e^{t}\sin(x)\cos(y),\quad x,y\in (0,2\pi),\label{BFD46}
\end{eqnarray}
is solution, adding an appropriate source term. We compute the errors
\begin{eqnarray*}
E_{2}(f)=||f_{N}(\cdot,T)-f(\cdot,T)||,\quad E_{\infty}(f)=|f_{N}(\cdot,T)-f(\cdot,T)|_{\infty},
\end{eqnarray*}
at $T=1$, for each of the functions $f=\zeta, u_{1}, u_{2}$ and approximations $f_{N}$ with $N=64$ (spatial stepsize $h=2\pi/N$). They are displayed in Tables \ref{BFD_t3} and \ref{BFD_t4}, for several time stepsizes $\Delta t$ and for $\zeta$ and $u_{1}$ components. (The results for $u_{2}$ are similar to those of $u_{1}$ and will not be shown here.) For the values of $\Delta t$ considered, the fixed point iteration (\ref{BFD45e}) converge; due to the regularity of the solution (\ref{BFD46}), spectral order of convergence in space is expected and, as shown in the tables, the second order of convergence in time is observed.
\begin{table}[htbp]
\begin{tabular}{c|c|c|c|c|}
$\Delta t$& $E_{2}(\zeta)$&Rate&$E_{2}(u_{1})$&Rate\\
\hline
$5\times 10^{-2}$&$6.0788\times 10^{-3}$&&$4.8431\times 10^{-4}$&\\
$2.5\times 10^{-2}$&$1.6290\times 10^{-3}$&$1.900$&$1.1566\times 10^{-4}$&$2.066$\\
$1.25\times 10^{-2}$&$4.1135\times 10^{-4}$&$1.986$&$2.8562\times 10^{-5}$&$2.018$\\
$6.25\times 10^{-3}$&$1.0305\times 10^{-4}$&$1.997$&$7.1183\times 10^{-6}$&$2.005$\\
$3.125\times 10^{-3}$&$2.5774\times 10^{-5}$&$1.999$&$1.7782\times 10^{-6}$&$2.001$\\
$1.5625\times 10^{-3}$&$6.4442\times 10^{-6}$&$2.000$&$4.4446\times 10^{-7}$&$2.000$\\
\hline
\end{tabular}
\caption{$L^{2}$ errors at $T=1$ and temporal convergence rates for the method (\ref{BFD43a})-(\ref{BFD44}) with respect to (\ref{BFD46}) with $N=64$, $a=c=-1/6, b=1/6, d=1/2$.\label{BFD_t3}}
\end{table}

\begin{table}[htbp]
\begin{tabular}{c|c|c|c|c|}
$\Delta t$& $E_{\infty}(\zeta)$&Rate&$E_{\infty}(u_{1})$&Rate\\
\hline
$5\times 10^{-2}$&$2.4253\times 10^{-2}$&&$9.6894\times 10^{-4}$&\\
$2.5\times 10^{-2}$&$6.5371\times 10^{-3}$&$1.891$&$2.3212\times 10^{-4}$&$2.062$\\
$1.25\times 10^{-2}$&$1.6535\times 10^{-3}$&$1.983$&$5.7396\times 10^{-5}$&$2.016$\\
$6.25\times 10^{-3}$&$4.1440\times 10^{-4}$&$1.997$&$1.4309\times 10^{-5}$&$2.004$\\
$3.125\times 10^{-3}$&$1.0366\times 10^{-4}$&$1.999$&$3.5748\times 10^{-6}$&$2.001$\\
$1.5625\times 10^{-3}$&$2.5919\times 10^{-5}$&$2.000$&$8.9354\times 10^{-7}$&$2.000$\\
\hline
\end{tabular}
\caption{$L^{\infty}$ errors at $T=1$ and temporal convergence rates for the method (\ref{BFD43a})-(\ref{BFD44}) with respect to (\ref{BFD46}) with $N=64$, $a=c=-1/6, b=1/6, d=1/2$.\label{BFD_t4}}
\end{table}
In the second example, we approximate line solitary wave solutions of (\ref{BFD1a}), (\ref{BFD1b}) of speed $c_{s}\neq 0$ of the form
\begin{eqnarray}
\zeta(x,y,t)&=&\zeta(x-c_{s}t-x_{0}),\nonumber\\
u_{1}(x,y,t)&=&0,\nonumber\\
u{2}(x,y,t)&=&u(x-c_{s}t-x_{0}),\label{BFD47}
\end{eqnarray}
with some $x_{0}\in\mathbb{R}$ and where $\zeta(x-c_{s}t-x_{0}), u(x-c_{s}t-x_{0})$ are therefore solitary wave solutions of the corresponding 1D B-FD systems, \cite{AnguloS2019}. The profiles $\zeta(X), u(X), X=x-c_{s}t-x_{0}$ are numerically generated by the procedure described in \cite{DDS3}. The ivp is first approximated by the periodic ivp on a long enough interval $\Omega_{L}=(-L,L)^{2}, L>0$, \cite{BChen}, and this is integrated numerically by the method proposed in this paper (involving a change of variable from $\Omega_{L}$ to $\Omega=(0,2\pi)$). For the B-FD system with $\gamma=0.8, x_{0}=-10, c_{s}=0.5$, $a=c=0, b=d=1/6$, and using the approximate profiles $\zeta(X), u(X)$ as initial condition, the code is run up to $T=40$ and some numerical results are reported.

We consider $L=32, N=256$ so that the space step sizes are $h_{x}=h_{y}=\frac{2L}{N}=0.25$. The $L^{2}$ and $L^{\infty}$ errors at $T=40$, with corresponding rates, for the $\zeta$ and $u_{2}$ components and several time step sizes, are displayed in Tables \ref{BFD_t5} and \ref{BFD_t5}, respectively. As in the previous experiment, the iteration (\ref{BFD45e}) was convergent for the values of $\Delta t$ considered. In addition, no further stability conditions were observed in the experiments.

\begin{table}[htbp]
\begin{tabular}{c|c|c|c|c|}
$\Delta t$& $E_{2}(\zeta)$&Rate&$E_{2}(u_{2})$&Rate\\
\hline
$5\times 10^{-2}$&$1.6051\times 10^{-3}$&&$4.6179\times 10^{-4}$&\\
$2.5\times 10^{-2}$&$4.0118\times 10^{-4}$&$2.0004$&$1.1542\times 10^{-4}$&$2.0003$\\
$1.25\times 10^{-2}$&$1.0029\times 10^{-4}$&$2.0001$&$2.8854\times 10^{-5}$&$2.0001$\\
$6.25\times 10^{-3}$&$2.5071\times 10^{-5}$&$2.0000$&$7.2133\times 10^{-6}$&$2.0000$\\
\hline
\end{tabular}
\caption{$L^{2}$ errors at $T=40$ and temporal convergence rates for the method (\ref{BFD43a})-(\ref{BFD44}) with respect to (\ref{BFD47}) in $\Omega_{32}$, with $N=256$, $a=c=0, b=d=1/6$.\label{BFD_t5}}
\end{table}

\begin{table}[htbp]
\begin{tabular}{c|c|c|c|c|}
$\Delta t$& $E_{\infty}(\zeta)$&Rate&$E_{\infty}(u_{2})$&Rate\\
\hline
$5\times 10^{-2}$&$1.2789\times 10^{-4}$&&$3.5570\times 10^{-5}$&\\
$2.5\times 10^{-2}$&$3.1845\times 10^{-5}$&$2.0057$&$8.8933\times 10^{-6}$&$1.9999$\\
$1.25\times 10^{-2}$&$7.9531\times 10^{-6}$&$2.0015$&$2.2233\times 10^{-6}$&$2.0000$\\
$6.25\times 10^{-3}$&$1.9878\times 10^{-6}$&$2.0004$&$5.5583\times 10^{-7}$&$2.0000$\\
\hline
\end{tabular}
\caption{$L^{\infty}$ errors at $T=40$ and temporal convergence rates for the method (\ref{BFD43a})-(\ref{BFD44})with respect to (\ref{BFD47}) in $\Omega_{32}$, with $N=256$, $a=c=0, b=d=1/6$.\label{BFD_t6}}
\end{table}

The evolution of the $\zeta$ component of the numerical solution is illustrated in Figure \ref{BFD_fig1}. The characteristic propagation as a traveling wave with permanent profile and constant speed is observed.
\begin{figure}[htbp]
\centering
\subfigure
{\includegraphics[width=0.45\columnwidth]{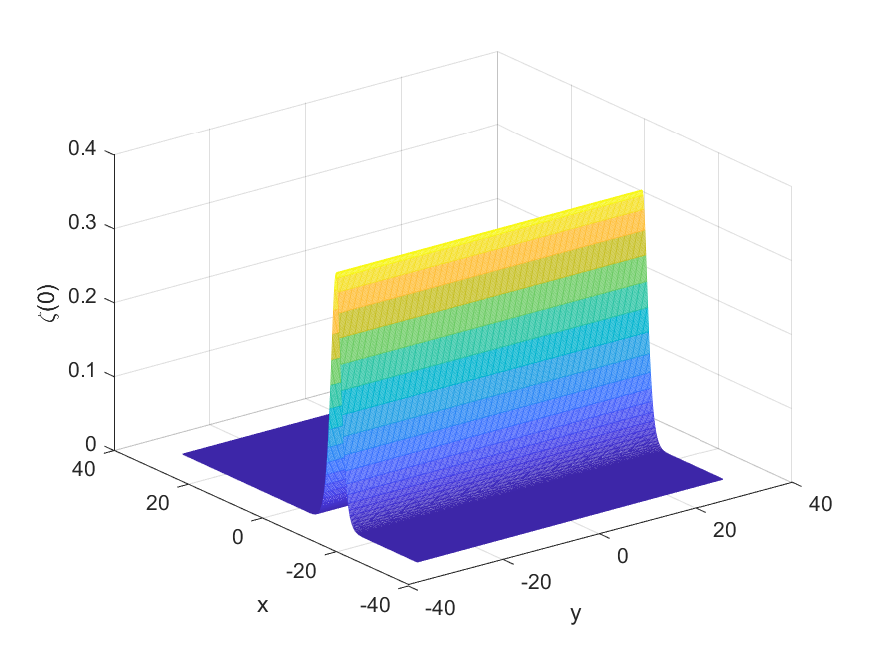}}
\subfigure
{\includegraphics[width=0.45\columnwidth]{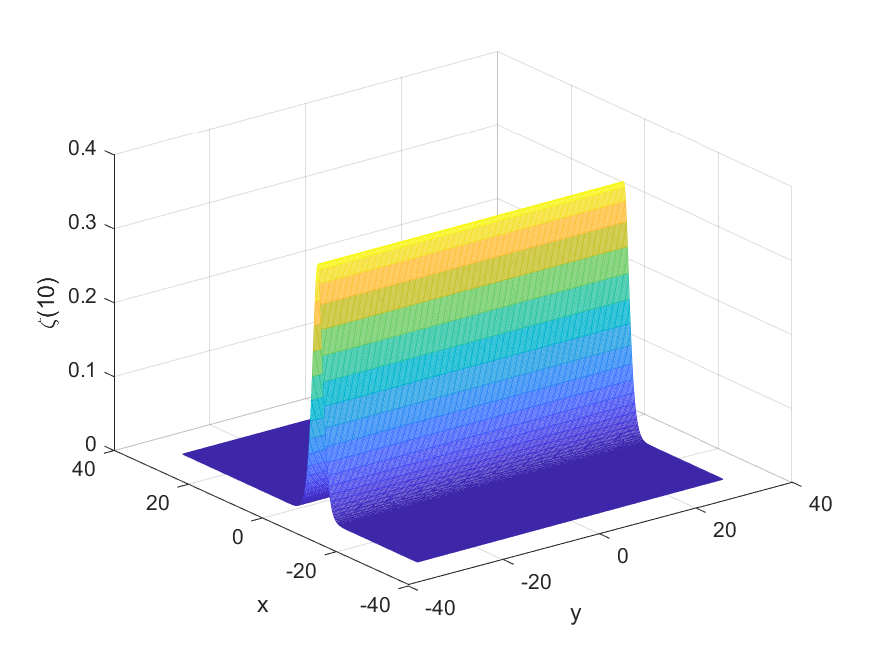}}
\subfigure
{\includegraphics[width=0.45\columnwidth]{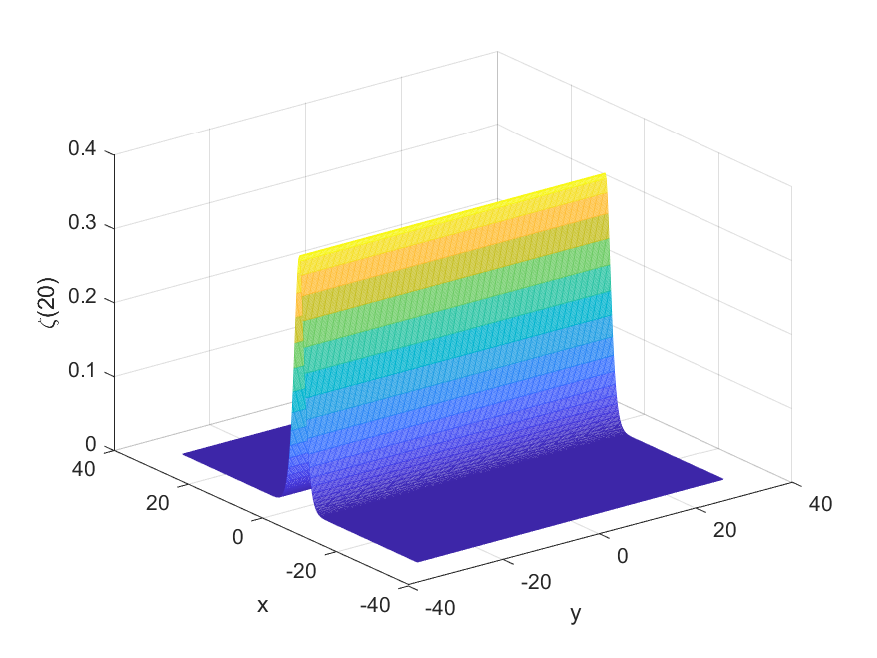}}
\subfigure
{\includegraphics[width=0.45\columnwidth]{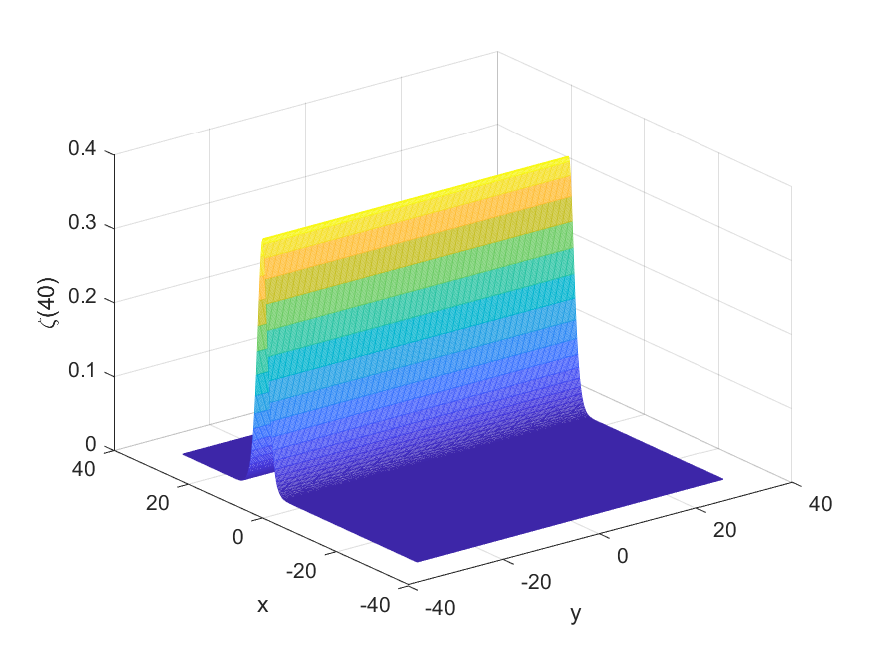}}
\caption{$\zeta$ component of the numerical solution at times $t=0,10,20,40$.}
\label{BFD_fig1}
\end{figure}
Figure \ref{BFD_fig2}(a) shows the corresponding cross sections at $y=0$, confirming the evolution of the 1D profile as solitary wave with speed $c_{s}$.

Some other results can measure the accuracy of the fully discrete scheme and are mentioned here by way of illustration. The time behaviour of the error in the amplitude of the $\zeta$ component is shown, in doble logarithmic scale, in Figure \ref{BFD_fig2}(b). Observe that after some growth up to approximately $t=10$, the errors tend to stabilize to a constant in time evolution for longer times. A similar behaviour is suggested in Figure \ref{BFD_fig2}(c), which shows the evolution of the discrete version of the Hamiltonian (\ref{BFD25}), computed using the spectral discretization as
\begin{eqnarray*}
H_{N}^{n}&=&\frac{h_{x}h_{y}}{2}\sum_{j,k}\left((1-\gamma)|Z_{j,k}^{n}|^{2}+\frac{1}{\gamma}|{\bf U}^{n}_{j,k}|^{2}\right.\\
&&\left.-\frac{1}{\gamma}Z_{j,k}^{n}|{\bf U}^{n}_{j,k}|^{2}-c(1-\gamma)|(\nabla_{N}Z^{n})_{j,k}|^{2}-\frac{a}{\gamma}\left((\nabla_{N}U_{1}^{n})_{j,k}^{2}+(\nabla_{N}U_{2}^{n})_{j,k}^{2}\right)\right.\\
&&\left.-\frac{1}{\gamma^{2}}\left((U_{1}^{n})_{j,k}(S_{1}^{N}U_{1}^{n})_{j,k}+(U_{2}^{n})_{j,k}(S_{1}^{N}U_{2}^{n})_{j,k}\right)-\frac{1}{\gamma^{3}}|(S_{1}^{N}{\bf U}^{n})|^{2}\right).
\end{eqnarray*}

\begin{figure}[htbp]
\centering
\subfigure
{\includegraphics[width=0.45\columnwidth]{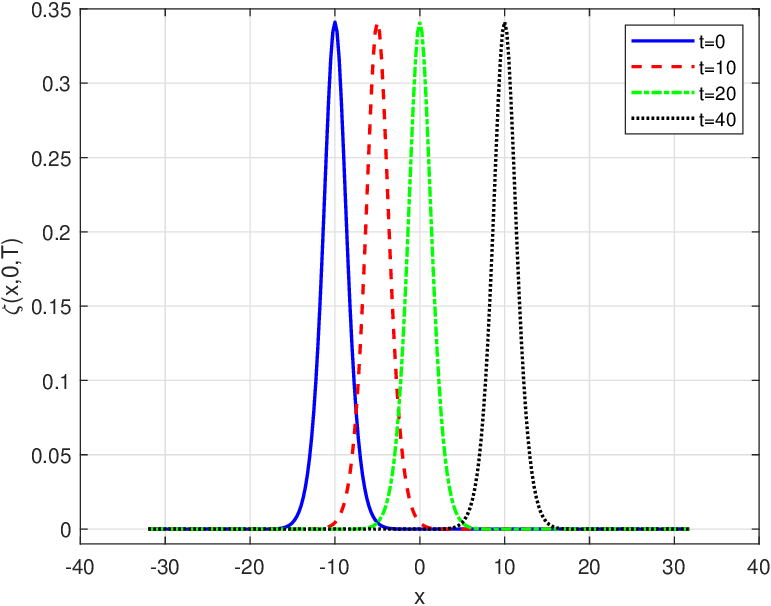}}
\subfigure
{\includegraphics[width=0.45\columnwidth]{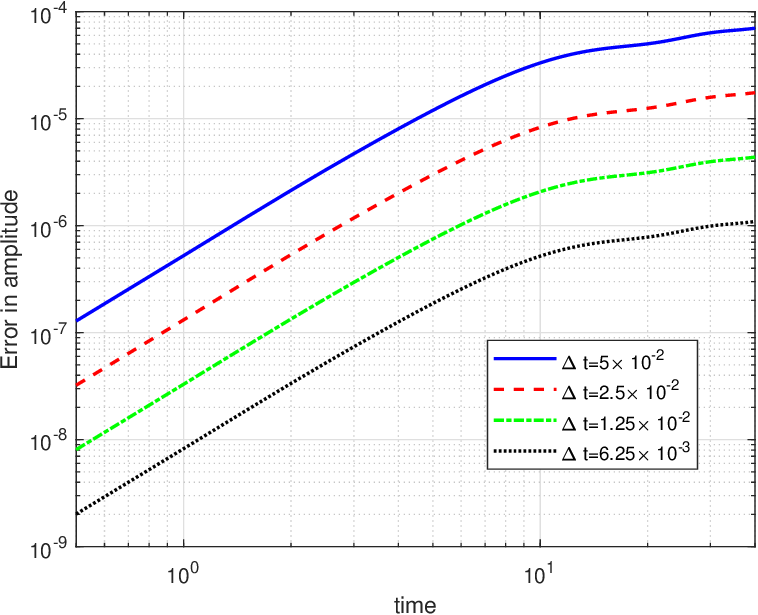}}
\subfigure
{\includegraphics[width=0.45\columnwidth]{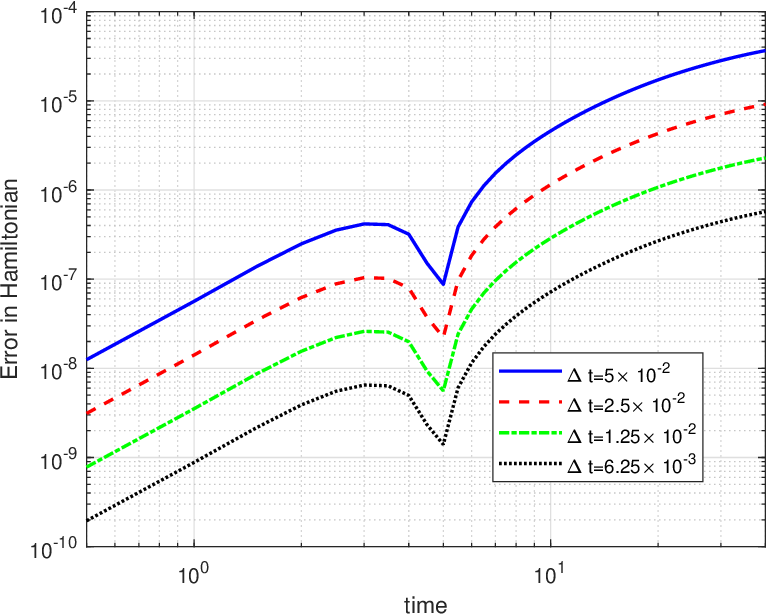}}
\caption{(a) $\zeta$ component of the numerical solution at times $t=0,10,20,40$ and $y=0$; (b) Time behaviour of the error in the amplitude for several time steps in loglog scale; (c) Time behaviour of the error in the Hamiltonian for several time steps in loglog scale.}
\label{BFD_fig2}
\end{figure}

\section*{Acknowledgments}
This research has been supported by Ministerio de Ciencia e Innovaci\'on project PID2023-147073NB-I00, {by Junta de Castilla y Le\'on under Project VA115P25 and by Department of Education of the Junta de Castilla y Le\'on and FEDER Funds under project CLU-2025-1-02- IMUVA}.

\end{document}